\documentclass{amsart}
\newcommand{\CVGM}{\mathrm{CV}_\mathrm{GM}}
\newcommand{\CVEM}{\mathrm{CV}_\mathrm{EM}}

\newcommand{\CVGMG}{\CVEM}

\newcommand{\leaf}{\mathrm{leaf}}
\newcommand{\internal}{\mathrm{internal}}
\newcommand{\vertex}{\mathrm{vertex}}
\newcommand{\rep}{\mathrm{rep}}

\usepackage{all}
\usepackage{xypic}

\begin{document}

\title{On the ideals of equivariant tree models}

\author[J.~Draisma]{Jan Draisma}
\address[Jan Draisma]{
Department of Mathematics and Computer Science\\
Technische Universiteit Eindhoven\\
P.O. Box 513, 5600 MB Eindhoven, Netherlands}
\thanks{The first author is supported by DIAMANT, an NWO
mathematics cluster.}
\email{j.draisma@tue.nl}

\author[J.~Kuttler]{Jochen Kuttler}
\address[Jochen Kuttler]{
Department of Mathematical and Statistical Sciences\\
632 Central Academic Building\\
University of Alberta\\
Edmonton, Alberta T6G 2G1,
CANADA}
\thanks{The second author is supported by an NSERC Discovery Grant.}
\email{jochen.kuttler@ualberta.ca}

\begin{abstract}
We introduce equivariant tree models in algebraic statistics, which
unify and generalise existing tree models such as the general Markov
model, the strand symmetric model, and group-based models such as the
Jukes-Cantor and Kimura models. We focus on the ideals of such models. We
show how the ideals for general trees can be determined from the ideals
for stars. A corollary of theoretical importance is that the ideal for a
general tree is generated by the ideals of its flattenings at vertices.
The main novelty is that our results yield generators of the full ideal
rather than an ideal which only defines the model set-theoretically.
\end{abstract}

\maketitle

\section{Set-up and theorems} \label{sec:Setup}

In phylogenetics, tree models have been introduced to describe the
evolution of a number of species from a distant common ancestor.
Given suitably aligned strings of nucleotides of $n$ species alive today,
one assumes that the individual positions in these strings have evolved
independently and according to the same statistical process. Counting and
averaging thus yields an empirical probability distribution on the set
$\{A,C,G,T\}^n$. On the other hand, any hypothetical evolutionary tree
whose leaves correspond to the $n$ species gives rise to a parameterised
family of probability distributions on $\{A,C,G,T\}^n$; see Section
\ref{sec:Statistics} for details. Here the parameters consist of an
initial distribution and transition matrices along the edges of the
tree. The challenge is to test whether the tree fits the data, that is,
whether the empirical distribution lies in the family. One way to do
this is the use of {\em phylogenetic identities}, equations that vanish
identically on the family. This paper is concerned with constructing
such identities for general trees from identities for smaller trees. The
set-up below unifies and generalises existing tree models in algebraic
statistics, while allowing for a clean and elegant treatment with methods
from classical invariant theory.  For more information on algebraic
statistics and its applications see \cite{Eriksson05, Pachter05} and
the many references there.

\begin{re}
The term {\em phylogenetic invariants} is commonly used for phylogenetic
identities. To avoid confusion with the invariants in
classical invariant theory, we do not use this term.
\end{re}

First, recall that a tree $T$ is a connected, undirected graph without
circuits; all our trees are finite and have at least two vertices. The
{\em valency} of a vertex of $T$ is the number of edges containing it. A
vertex of $T$ is called a leaf if it has valency $1$, and an internal
vertex otherwise; if $p,q$ are vertices, we write $p \sim q$ if there is
an edge connecting them. We write $\vertex(T),\leaf(T),\internal(T)$
for the sets of vertices, leaves, and internal vertices of $T$,
respectively. Stars are trees of diameter at most $2$, and a centre
of a star is a vertex at distance $1$ to all other vertices---so if
the star has more than $2$ vertices, then its centre is unique. A {\em
subtree} of a tree is a connected induced subgraph, and a {\em substar}
is a subtree which itself is a star. So the map that sends a vertex $p$
of $T$ to the induced subgraph on $p$ together with its neighbours is
a bijection between $\vertex(T)$ and the set of substars of $T$, except
when $T$ consists of a single edge. In all that follows, we work over a
ground field $K$ that is algebraically closed and of characteristic zero.

\begin{de}
A {\em spaced tree} $T$ is given by the following data: First, a finite
undirected tree, also denoted $T$; second, for every $p\in \vertex(T)$ a
finite-dimensional vector space $V_p$; third, a non-degenerate symmetric
bilinear form $(. \mid .)_p$ on each $V_p$; and fourth, for every $p \in
\internal(T)$ a distinguished basis $B_p$ of $V_p$ which is orthonormal
with respect to $(. \mid .)_p$.

The space $V_p$ at a leaf $p$ may also be given a distinguished basis
$B_p$, orthonormal with respect to $(. \mid .)_p$, in which case $p$
is called a {\em based leaf}. An internal vertex of $T$ will also be
called based. Any subtree of the underlying tree of $T$ is regarded as
a spaced tree with the data that it inherits from $T$.
\end{de}

Note that there is some redundancy in this definition: given the
distinguished basis $B_p$ at a based vertex $p$ one could {\em define}
$(. \mid .)_p$ by the requirement that $B_p$ be orthonormal. We will
leave out the subscript $p$ from the bilinear form when it is obvious
from the context. In many applications in algebraic statistics, symmetry
is imposed on the algebraic model. This notion is captured well by
the following notion of a $G$-spaced tree. Fix, once and for all,
a finite group $G$.

\begin{de}
A {\em $G$-spaced tree} (or $G$-tree, for short) is a spaced tree
$T$ in which the space $V_p$ at every vertex $p$ is a $G$-module, on
which $(. \mid .)_p$ is $G$-invariant, and in which $B_p$ is $G$-stable
whenever $p$ is a based vertex. Any subtree of the underlying tree of $T$
is regarded as a $G$-spaced tree with the data that it inherits from $T$.
\end{de}

The objects that parameterise probability distributions in the algebraic
model are what we propose to call representations of spaced trees.

\begin{de}
Let $T$ be a spaced tree. A {\em representation} of $T$ is a collection of
tensors $(A_{qp})_{q \sim p} \in V_q \otimes V_p$ along the edges of $T$
with $A_{pq}=A_{qp}^t$, where $\cdot^t$ is the natural isomorphism $V_p
\otimes V_q \rightarrow V_q \otimes V_p$. The space of all representations
of $T$ is denoted $\rep(T)$. A representation of $T$ gives rise to a
representation of any subtree of $T$.

For a $G$-spaced tree $T$, a {\em $G$-representation} or {\em equivariant
representation} of $T$ is a representation $(A_{qp})_{p \sim q}$ where
each $A_{qp}$ is a $G$-invariant element in the $G$-module $V_q \otimes
V_p$. The space of such representations is denoted $\rep_G(T)$.
A $G$-representation of $T$ gives rise to a $G$-representation of any
subtree of $T$.
\end{de}

Using the bilinear form $(.|.)_p$ we may identify $V_p$ with its dual
$V_p^*$, and hence $V_q \otimes V_p$ with $V_q \otimes V_p^* \cong
\Hom(V_p,V_q)$. Thus viewing $A_{qp}$ as a linear map $V_p \rightarrow
V_q$ and, similarly, $A_{pq}$ as a linear map $V_q \rightarrow V_p$, the
condition $A_{pq}=A_{qp}^t$ translates into $(A_{qp} v \mid w)_q=(v \mid
A_{pq} w)_p$ for all $v \in V_p$ and $w \in V_q$. Put yet differently,
if $p$ and $q$ are both based, then this says that the matrix of
$A_{pq}$ relative to the bases $B_p$ and $B_q$ is the transpose of
the matrix of $A_{qp}$. In the applications to statistics, the spaces
$\Hom(V_p,V_q)$ or the space of $|B_q| \times |B_p|$-matrices are perhaps
more natural to work with than $V_q \otimes V_p$, as the elements of
a representation correspond to {\em transition matrices}; see Section
\ref{sec:Statistics}. However, there are good reasons to work with
$V_q \otimes V_p$; for instance, the correct action of $\lieg{GL}(V_p)
\times \lieg{GL}(V_q)$ on the edge parameters turns out to be the natural
action on $V_q \otimes V_p$ rather than that on $\Hom(V_q,V_p)$; see
Lemma~\ref{lm:GTAction} and the proof of Proposition \ref{prop:Dense1}.
Finally we note that if $T$ is a $G$-tree, then by the invariance of
the bilinear form $(.|.)_p$ the identifications above still make sense when
passing to $G$-invariant elements: $(V_q \otimes V_p)^G=\Hom_G(V_p,V_q)$,
etc.

Notice the slight discrepancy between our notion of representations and
the notion in quiver representation theory, where---apart from the fact
that the underlying graph is directed---the spaces $V_p$ form part of
the data comprising a representation.

A {\em $T$-tensor} is any element of $\bigotimes_{p \in \leaf(T)} V_p$,
which space we will denote by $L(T)$ throughout the text.  $T$-tensors
correspond to marginal probability distributions in statistics; see
Section \ref{sec:Statistics}. An important operation on spaced trees,
representations, and $T$-tensors is $*$, defined as follows. Given $k$
spaced trees $T_1,\ldots,T_k$ whose vertex sets share a common based
leaf $q$ with common space $V_q$ and common basis $B_q$ but which
trees are otherwise disjoint, we construct a new spaced tree $*_i T_i$
obtained by gluing the $T_i$ along $q$, while the space at a vertex $p$
of $*_i T_i$ coming from $T_i$ is just the space attached to it in $T_i$,
with the same distinguished bilinear form, and the same basis if $p$
is based. Given representations $A_i \in \rep(T_i)$ for $i=1,\ldots,k$,
we write $*_i A_i$ for the representation of $*_i T_i$ built
up from the $A_i$.  Now let $\Psi_i$ be a $T_i$-tensor, for all $i$.
Then we obtain a $T$-tensor by tensoring as follows:
\[ *_i \Psi_i:=\sum_{b \in B_q} \otimes_i (b \mid \Psi_i), \]
where we abuse the notation $(. \mid .)$ for the natural contraction
\[ V_q \times \bigotimes_{p \in \leaf(T_i)} V_p \to
	\bigotimes_{p \in \leaf(T_i) \setminus \{q\}} V_p \]
determined by the bilinear form $(. \mid .)_q$. Notice that this $*$
operator is {\em not} a binary operator extended to several factors;
nevertheless, when convenient, we will write $T_1 * \cdots * T_k$ for
$*_i T_i$ and $\Psi_1 * \cdots * \Psi_i$ for $*_i \Psi_i$.

Now we come to a fundamental procedure that associates to any
representation of a spaced tree $T$ a $T$-tensor. Let $A \in \rep(T)$.
We proceed inductively. First, if $T$ has a single edge $\{p,q\}$, then
$\Psi_T(A):=A_{qp}$, regarded as an element of $L(T)=V_q \otimes V_p$.
If $T$ has more than one edge, then let $q$ be any internal vertex of
$T$. We can then write $T=*_{p \sim q} T_p$, where $T_p$ is the
{\em branch of $T$ around $q$ containing $p$}, constructed by taking the
connected component of $T - q$ (the graph obtained from $T$ by removing $q$ and all edges attached to $q$) containing $p$, and reattaching $q$ to $p$.

The representation $A$ induces representations $A_p$ of the
$T_p$, and by induction $\Psi_{T_p}(A_p)$ has been defined. We now set
\[ \Psi_T(A):=*_{p \sim q} \Psi_{T_p}(A_p). \]
A straightforward proof by induction shows that this is independent of
the choice of $q$ and that this formula is also valid if $q$ is actually
a leaf. Now we can define the key objects of this paper.

\begin{de}
Let $T$ be a spaced tree. The {\em general Markov model} associated to
$T$ is the algebraic variety
\[ \CVGM(T):=\overline{\{\Psi_T(A) \mid A \in \rep(T) \}}
\subseteq L(T), \]
where the closure is taking in the Zariski topology.

Similarly, for a $G$-spaced tree $T$, the {\em equivariant
model} associated to $T$ is the algebraic variety
\[ \CVEM(T):=\overline{\{\Psi_T(A) \mid A \in \rep_G(T) \}} \subseteq
L(T). \]
\end{de}

Notice that {\em a priori} both the individual tensors $\Psi_T(A)$ and the
varieties $\CVGM(T)$, $\CVEM(T)$ depend on the bases $B_q$ at internal
vertices $q$. This is only natural, as in applications these bases have
an intrinsic meaning; see Section \ref{sec:Statistics}.  However, more
can be said about this dependency; see Lemma \ref{lm:GTAction}.

To streamline our discussion, we will consider $\CVGM$ as the special case
of $\CVEM$ where $G$ is trivial. An important goal in algebraic statistics
is finding the ideal of all polynomials on the space $L(T)$ that vanish
on $\CVEM(T)$. Our first result is a procedure for constructing these
ideals from the ideals for substars of $T$.

\begin{thm} \label{thm:Procedure}
For any $G$-spaced tree $T$, the ideal $I(\CVEM(T))$ can be expressed
in the ideals $I(\CVEM(S))$ where $S$ runs over the $G$-spaced substars
of $T$ with at least three leaves. In particular, for any spaced tree,
the ideal of $I(\CVGM(T))$ can be expressed in the ideals $I(\CVGM(S))$
where $S$ runs over the spaced substars of $T$ with at least three leaves.
\end{thm}

This theorem is admittedly formulated somewhat vaguely. However, its
proof in Section \ref{sec:Proofs} gives rise to the explicit, recursive
Algorithm \ref{alg:Procedure} for determining $I(\CVEM(T))$ from the
ideals $I(\CVEM(S))$; this justifies the present formulation.

We now present a variant of Theorem \ref{thm:Procedure} which is
perhaps less useful for actual computations, but which is of fundamental
theoretical interest. This variant uses a second important operation on
spaced trees and leaf tensors, namely, {\em flattening}. Fix any vertex
$q$ in a spaced tree $T$, and define an equivalence relation on $\leaf(T)
\cup \{q\}$ by $p \cong r$ if and only if either $p=q=r$ or $p,r \neq q$
lie in the same connected component of $T-q$. Construct a spaced star
$\flat_q T$ as follows: First, the vertex set is the set $\leaf(T) \cup
\{q\}/\cong$ of equivalence classes, and the class of $q$ is attached
to all other classes by an edge.  To the class $C$ we attach the space
$V_C:=\bigotimes_{p \in C} V_p$ equipped with the bilinear form inherited
from the $V_p$, and if all $p \in C$ are based, then $C$ is based with
the tensor product of the bases $B_p$. This new spaced tree $\flat_q T$
is called the {\em flattening} of $T$ at $q$. Note that we allow $q$ to
be a leaf of $T$, in which case $\flat_q T$ has a single edge. The space
$L(T)$ of $T$-tensors is naturally identified with the space $L(\flat_q
T)$ of $\flat_q T$-tensors, and expanding the definition of $\Psi_T$
at $q$ one readily finds that
\[ \CVEM(T) \subseteq \CVEM(\flat_q T) \text{ for all } q. \]
Our second main result shows that this characterises $\CVEM(T)$.

\begin{thm} \label{thm:Flattening}
For any $G$-spaced tree $T$ we have
\[ I(\CVEM(T))=\sum_{q \in \vertex(T)} I(\CVEM(\flat_q T)).
\]
\end{thm}

\begin{re}
\begin{enumerate}
\item If $T$ has more than one edge, then it suffices to let $q$ run over
$\internal(T)$.
\item To avoid confusion we stress that $\flat_q T$ is {\em not} a
substar of $T$, unless $T$ itself is a star with centre $q$, in which
case $\flat_q T \cong T$.
\end{enumerate}
\end{re}

Many special cases of our main results are known in the literature. In
particular \cite{Allman04} contains set-theoretic versions of our theorems
for the general Markov model, and poses Theorem \ref{thm:Flattening}
for the general Markov model as Conjecture 5. In
\cite{Sturmfels05b}
the ideals of equivariant models with $G$ abelian and all $V_p$ equal to
the regular representation $KG$ are determined, following ideas from \cite{Evans93}. An important
observation that makes this feasible is that these varieties are toric;
see also Section \ref{sec:Toric}.
Some more specific references to the literature may be found in the Section~\ref{sec:Statistics}, which explains the relevance of spaced trees and their representations to statistics.
After that, in Section \ref{sec:Multiplying} we prove a
key tool on multiplying varieties of matrices, which we then use in
Section \ref{sec:Proofs} to prove our main results. Finally, Section
\ref{sec:Toric} contains a result on toricness of certain abelian
equivariant models.

\section{Acknowledgments}
The first author thanks Seth Sullivant for his great EIDMA/DIAMANT course
on algebraic statistics in Eindhoven. It was Seth who pointed out that a
result like the one in Section \ref{sec:Multiplying} could be used to
treat various existing tree models in a unified manner.

\section{Relevance to statistics} \label{sec:Statistics}

In the applications of our results to algebraic statistics, the spaced
tree $T$ that we start with only has based vertices. Indeed, the
bases $B_p$ have some physical meaning. In phylogenetics, for instance,
they are usually all equal to $\{A,C,G,T\}$, the building bricks for
DNA. Furthermore, an internal vertex $r$ is singled out as root, and
the base field is $K:=\CC \supseteq \RR$.  An element of $V_p$ which on
the basis $B_p$ has non-negative real coefficients that add up to $1$ is
regarded as a probability distribution on $B_p$; together they form the
{\em probability simplex} $\Delta(V_p) \subseteq V_p$. A representation
of $T$ is called {\em stochastic} if all maps $A_{qp}:V_p \rightarrow
V_q$ directed away from $r$ satisfy $A_{qp} \Delta(V_p) \subseteq
\Delta(V_q)$, which amounts to saying that the entries of
$A_{qp}$, regarded as a matrix relative to the bases $B_q$
and $B_p$, are
real and non-negative and that $A_{qp}$ has all column sums
equal to $1$. A root
distribution $\pi \in \Delta(V_r)$ and a stochastic representation
$A$ of $T$ determine a probability distribution on $\prod_{p \in
\vertex(T)} B_p$ and, by taking marginals, a distribution on $\prod_{p
\in \leaf(T)} B_p$, which can be thought of as an element $\Phi_T(A,\pi)$
of $\Delta(L(T))$. Write $T=T_1*\cdots*T_k$
at $r$ and let $A_1,\ldots,A_k$ be the induced representations on the
$T_i$. Then the distribution is
\[ \Phi_T(A,\pi)=\sum_{b \in B_r} (b \mid \pi) (b \mid \Psi_{T_1}(A_1))
\otimes \cdots \otimes (b \mid \Psi_{T_k}(A_k)), \]
which equals $\Psi_T(A')$, where $A' \in \rep(T)$ is the (non-stochastic)
representation obtained from $A$ by composing a single $A_{pr}$ leading
away from the root with the diagonal linear map $V_r \rightarrow V_r$
determined by $b \mapsto \pi(b) b$. We define the set
\[ \CVGM(T,r):=\{ \Phi_T(A,\pi) \mid \pi \in \Delta(V_r) \text{ and } A
\in \rep(T) \text{ stochastic}.\} \]
A natural equivariant analogue of this for a $G$-tree $T$ is
\[ \CVEM(T,r):=\{ \Phi_T(A,\pi) \mid \pi \in \Delta(V_r)
\text{ $G$-invariant and } A \in \rep_G(T) \text{ stochastic}\}, \]
but as the following examples from phylogenetics show it also makes sense
to allow for arbitrary root distributions rather than $G$-invariant ones;
see below how to handle these.

\begin{ex}\label{ex:Models}
In all models below, the $B_p$ are all equal to $\{A,C,G,T\}$ and
are all equipped with the same permutation action of some $G$. Recall
that the nucleotides fall into two classes of bases, according to their
chemical structure: the purines Adenine and Guanine and the pyrimidines
Cytosine and Thymine. This explains some of the choices in the following
models. All of them are equivariant models in our sense. The labels of
these models are those used in \cite{Pachter05}.

\begin{enumerate}
\item In the {\em Jukes-Cantor model JC69} $G=\Sym(\{A,C,G,T\})$ (or
the alternating group, which has exactly the same equivariant maps $V_p
\rightarrow V_p$). One assumes a $G$-invariant root distribution---which
in this case means that it is uniform.

\item In the {\em Kimura model K80} $G$ is the dihedral group generated
by $(A,C,G,T)$ and $(A,G)$. It is the group of symmetries of
the following square.
\[ \begin{matrix}
	A & - & C\\
	| &     & |\\
	T & - & G
\end{matrix} \]
Again, the root distribution is taken $G$-invariant, which means uniform.

\item In the {\em Kimura model K81} $G$ is the Klein $4$-group and the
root distribution is $G$-invariant (uniform).

\item In the {\em strand-symmetric model CS05} $G$ generated by
the transpositions $(A,G)$ and $(C,T)$ and the root distribution is
$G$-invariant.

\item In the {\em HKY85 model} $G$ is as in the strand-symmetric model,
but one allows for non-$G$-invariant root distributions.

\item In the {\em Felsenstein model F81} $G$ is the full symmetric
(or alternating) group, and the root distribution arbitrary.
\end{enumerate}

The ideals of all these models were determined in
\cite{Casanellas05,Sturmfels05b}. Moreover, \cite{Casanellas07} gives
local equations at biologically meaningful points.
\end{ex}

\begin{re}
A similar construction of tree models appears in \cite{Buczynska07}. There
the spaces at all vertices are required to be the same space $W$,
and the tensors at the edges are allowed to vary in some fixed
subspace $\widehat{W}$ of $W \otimes W$ consisting of {\em symmetric}
tensors. Otherwise the construction of the model is the same. We should
mention that not all models obtained in this manner fit within our
framework. For instance, the model where $\widehat{W}$ is the entire space
of symmetric tensors cannot be characterised as the set of $G$-invariant
tensors in $W \otimes W$ for some group $G$ acting on $W$. Our present
approach does not apply to this setting.
\end{re}

Similar to the observations in \cite{Allman04}, as a consequence of the construction of $\CVEM(T)$, it is a closed cone (i.e.\ invariant under scalar multiplication in $L(T)$) and therefore uniquely defines a projective variety in $\PP(L(T))$, denoted $\PP(\CVEM(T))$, and defined by the same ideal as $\CVEM(T)$. Notice that because the elements of $\CVEM(T,r)$ have coordinate sum equal to $1$, $\CVEM(T,r)$ actually maps injectively into $\PP(L(T))$.
The following proposition justifies our quest for the ideal $I(\CVEM(T))$:
it contains all homogeneous polynomials vanishing on the statistically
meaningful set $\CVEM(T,r)$.

\begin{prop} \label{prop:Dense1}
Provided that all $V_p$ are non-zero, the image of the set $\CVEM(T,r)$
in $\PP(L(T))$ is Zariski dense in the variety $\PP(\CVEM(T))$.
\end{prop}

\begin{proof}
First, the set of root distributions on $B_r$ is clearly Zariski-dense
in the set of all (complex) $\pi \in V_r$ with $\sum_{b \in B_r} \pi(b)=1$.
Similarly, for a single edge $pq$ pointing away from $r$, the stochastic
matrices in $\Hom_G(V_p,V_q)$ are Zariski dense in the complex matrices
in $\Hom_G(V_p,V_q)$ with column sums $1$. This follows from an
explicit parameterisations of such equivariant stochastic matrices $A$:
for every $b \in B_p/G$ the $b$-th column of $A$ varies in a certain
(scaled probability) simplex of dimension $|B_q/G_b|-1$, where $G_b$
is the stabiliser of $b$ in $G$. This simplex is dense in the subset of
$V_q^{G_b}$ where the sum of the coordinates is $1$.

Next we claim that for $A$ in an open dense subset of $\rep_G(T)$ we
can write $\Psi_T(A)$ as $s \Phi_T(A',\pi)$ for some $A' \in \rep_G(T)$
having column sums $1$ and some $\pi$ with $\sum_b \pi(b)=1$.  To see
this, first fix a vertex $q$ and take for every vertex $p \sim q$ a
copy of the torus $(K^*)^{B_q/G}$, considered as the diagonal subgroup
of $\GL(V_q)$ acting by multiplication by a scalar on the span of each
$G$-orbit on $B_q$ and hence centralising $G$ on $V_q$. The copy for
$p$ acts on $\Hom(V_p,V_q)$ by $g(A):=gA$ and on $\Hom(V_q,V_p)$ by
$g(A):=Ag$. Note that this latter action is {\em not} the natural one
on $\Hom(V_q,V_p)$, in which $g$ would be replaced by its inverse, but
that this action {\em is} the natural one on $V_q \otimes V_p$; see also
Lemma \ref{lm:GTAction} and the remarks preceding it.  In particular,
this action maps representations to representations. A straightforward
computation shows that the subtorus
\[ H_q:=\{ (c_{p,b})_{b \in B_q/G, p \sim q} \in
	\prod_{p \sim q} (K^*)^{B_q/G}
\mid \prod_{p \sim q} c_{p,b} = 1 \text{ for all } b \in B_q \} \]
leaves $\Psi_T$ invariant. Now let $A$ be any $G$-representation of $T$
such that all column sums of all $A_{qp}$ directed away from $r$ are
non-zero; this is an open dense condition on $A$.  Given any non-root
vertex $q$, by acting with $H_q$ we can achieve that the $A_{qp}$
leading away from $r$ have column sums $1$, while the map $A_{qp}$,
where $p$ is the parent of $q$ relative to $r$, may not. If we do this
for all non-root vertices in a bottom-up manner, and finally also for
$H_r$, then we achieve that all $A_{qp}$ leading away from $r$ have
column sums $1$, except for a single $A_{pr}$; note that we have not
altered $\Psi_T(A)$ in this process. Denote the column sums of $A_{pr}$
by $(\sigma_b)_{b \in B_r}$. Dividing column $b$ of $A_{pr}$ by $\sigma_b$
gives a representation $A'$ all of whose matrices leading away from $r$
have column sums $1$. Also, for $A$ in an open dense subset, $\sum_b
\sigma_b=:s$ is non-zero, and dividing $\sigma$ by $s$ gives a $\pi$
adding up to $1$ such that $s \Phi(A',\pi)=\Psi(A)$.  This proves the
claim, and hence the proposition.
\end{proof}

As we saw in the examples above, one may want to allow arbitrary root
distributions, which are not necessarily $G$-invariant. More generally,
one might want to allow the root distribution to vary in a certain
self-dual submodule of $V_r$, and this would require only minor changes
in the discussion that follows---but here we concentrate on the situation
where all elements of (the probability simplex in) $V_r$ are allowed. We
define the set
\[ \CVEM(T,V_r):=\{ \Phi_T(A,\pi) \mid A \in \rep_G(T) \text{ stochastic and }
	\pi \in \Delta(V_r) \}. \]
One can elegantly describe $\CVEM(T,V_r)$ as follows. Let $T'$ be the
spaced tree obtained from $T$ by connecting a new vertex $r'$ to the root
$r$ of $T$ and attaching to $r'$ the $G$-module $V_{r'}:=V_r$, endowed
with the same bilinear form. Then $L(T') = V_{r'}\otimes L(T)$, and
since $V_{r'} \cong V_{r'}^*$ we may think of $\CVEM(T')$ as a subset
of $\Hom(V_{r'},L(T))$. Since it also consists of $G$-fixed points,
and since this identification is $G$-equivariant, it is a subset of
$\Hom_G(V_{r'},L(T))$.

\begin{prop} \label{prop:Dense2}
The image of $\CVEM(T,V_r)$ in $\PP(L(T))$
is a Zariski dense subset in the projective variety associated to the cone
\[ \overline{\CVEM(T')V_{r'}}, \]
where we regard $\CVEM(T')$ as a subset of
$\Hom_G(V_{r'},L(T))$. The ideal of this cone can be determined from
the ideal of $\CVEM(T')$.
\end{prop}

\begin{proof}
We have
\[\Phi_T(A,\pi)=\Psi_{T'}(A')\pi,\]
where $A'$ is obtained from $A$ by putting the identity $I_{V_r}$
along the edge $rr'$. This shows that $\Phi_T(A,\pi)$ is contained
in $\CVEM(T')V_{r'}$. For the converse we reason as before: for $A'$
in an open dense subset of $\rep_G(T')$ we can write $\Psi_{T'}(A')$
as $\Psi_{T'}(A)$ where all $A_{qp}$ directed away from $r$ have column
sums equal to $1$ except possibly for $A_{r'r}$. We have
\begin{align*} \Psi_{T'}(A)\pi&=
\sum_{b \in B_r} (A_{r'r} b | \pi)
(b \mid \Psi_{T_1}(A_1)) \otimes \ldots \otimes
(b \mid \Psi_{T_m}(A_m)) \\
&=
\sum_{b \in B_r} (b | A_{r'r} \pi)
(b \mid \Psi_{T_1}(A_1)) \otimes \ldots \otimes
(b \mid \Psi_{T_m}(A_m)) \\
&=
\Phi_T(A,A_{rr'} \pi)=s\Phi_T(A,\pi'),
\end{align*}
where $s$ is taken such that $\pi':=s^{-1} A_{r'r}\pi$ has $\sum_b
\pi'(b)=1$.

The proof of the last statement is deferred to the end of Section~\ref{sec:Proofs}.
\end{proof}

\section{Multiplying varieties of matrices}
\label{sec:Multiplying}

In this section we derive a key tool that will be used in the proofs
of our results. As before let $K$ be an algebraically closed field of
characteristic $0$, and let $t$ be a natural number. For $\bk,\bl \in
\NN^t$ let $M_{\bk,\bl}$ denote the space $M_{k_1,l_1} \times \cdots
\times M_{k_t,l_t}$, where $M_{k,l}$ is the space of $p \times q$-matrices
over $K$. To formulate and prove our results in their full strength,
it is convenient to use some notions from the language of schemes, for
which we refer to \cite{Hartshorne77}. The main point here is that we
do not require ideals to be radical.

Recall that if $X$ is an affine variety, then a closed subscheme $S$
of $X$ is given by an ideal $I$ of the ring $K[X]$ of regular functions
on $X$: the underlying closed subset of $X$ is the set of zeros of $I$,
and the $K$-algebra associated to $S$ is $K[X]/I$. We write $I(S)$
for the ideal of the subscheme $S$.  If $X$ is a variety on which some
group $\Gamma$ acts, then $S$ is called a $\Gamma$-subscheme if and
only if $I(S) \subseteq K[X]$ is $\Gamma$-stable.  Finally, if $f \colon X
\to Y$ is a map between varieties, inducing the pull back homomorphism
$f^\sharp \colon K[Y] \to K[X]$, and if $S \subseteq X$ is a subscheme,
then the \emph{image scheme} of $S$ is defined as the scheme theoretic
closure of $f(S)$, i.e.\ the subscheme of $Y$ defined by the ideal
$(f^\sharp)^{-1}(I(S))$. By slight abuse of notation it is usually denoted
as $f(S)$. It is clear that if $f$ is $\Gamma$-equivariant for some group
$\Gamma$ acting on $X$ and $Y$, then the image of a $\Gamma$-subscheme
is again a $\Gamma$-subscheme. Also notice that if $S$ is a subvariety,
i.e., if $I(S)$ is radical, then so is $f(S)$---it is precisely the
Zariski closure of the set-theoretic image of $S$ under $f$.

We now specialise to multiplying schemes of matrices. We write $\mu \colon
M_{\bk,\bl} \times M_{\bl,\bm} \to M_{\bk,\bm}$ for the multiplication
and $\mu^\sharp$ for the co-multiplication, $\mu^\sharp(f)(A,B) =
f(A\cdot B)$. Given two subschemes $V \subseteq M_{\bk,\bl}$ and $W\subseteq
M_{\bl,\bm}$, put
\[ V \cdot W := \mu(V,W). \]
If $V,W$ are subvarieties this is just
\[\overline{\{A B \mid A \in V, B \in W\}},\]
where $AB:=(A_1B_1, \dots,  A_tB_t)$. In general, the underlying
topological space is still the closure of the set $\mu(V,W)$, but the
ideal is $(\mu^\sharp)^{-1}(I(V \times W))$.

The operation $\cdot$ is associative in the sense that, given a third
subscheme $U \subseteq M_{\bj,\bk}$, one has
\[ U \cdot (V \cdot W)=(U \cdot V) \cdot W\]
we therefore simply write $U \cdot V \cdot W$.

Let $\GL_{\bl} = \GL_{l_1} \times \GL_{\l_2} \times \dots \times \GL_{l_t}
\subseteq M_{\bl}$.  Frequently the subschemes we are interested in will
be invariant by left- or right-multiplication by $\GL_{\bl}$. In this
context it is worth mentioning that a subscheme $V
\subseteq M_{\bk,\bl}$
is a $\GL_{\bl}$-subscheme, i.e., stable by right-multiplication, if
and only if $V \cdot M_{\bl,\bl} = V$. This follows from the fact that
for any $K$-algebra $A$ and any nonconstant $f \in K[x_1,\dots,x_m]$,
$\Spec A[x_1,\dots,x_m]_f$ is dense in $\Spec A[x_1,\dots,x_m]$, and
thus $V \cdot \GL_\bl$ is (scheme theoretically) dense in $V \cdot M_\bl$.

\begin{thm} \label{thm:VW}
Let $\bk,\bl,\bm \in \NN^t$ and let $V$ and $W$ be subschemes of
$M_{\bk,\bl}$ and $M_{\bl,\bm}$, respectively. Then we have
\begin{equation} \label{eq:VW}
I(V \cdot M_{\bl,\bl} \cdot W) = I(V \cdot M_{\bl,\bm}) +
I(M_{\bk,\bl} \cdot W).
\end{equation}
\end{thm}

In the following proof we will use the First Fundamental Theorem
in invariant theory, which describes the invariant polynomials of
$\lieg{GL}_n$ on a direct sum of copies of $K^n$ and the dual space
$(K^n)^*$. This theorem is due to Weyl \cite{Weyl39}.  Another tool from
representation theory is the Reynolds operator: when a reductive group
acts rationally on a vector space, then the Reynolds operator $\rho$
is the projection onto the invariant vectors with kernel the direct
sum of all non-trivial irreducible submodules.
Modern treatments on
invariant theory are \cite{Derksen02,Goodman98,Kraft96}.

\begin{proof}
The inclusion $\supseteq$ is obvious. To prove the opposite inclusion we
first replace $V$ by $V \cdot M_{\bl,\bl}$ and $W$ by $M_{\bl,\bl} \cdot
W$---this clearly keeps invariant both sides of
\eqref{eq:VW}, and moreover turns $V$ and $W$ into $\lieg{GL}_{\bl}$-subschemes, where
$\lieg{GL}_\bl:=\lieg{GL}_{l_1} \times \cdots \times \lieg{GL}_{l_t}$, which acts on $M_{\bk,\bl} \times M_{\bl,\bm}$
by \[(g_1,\dots,g_t)((A_1,\dots,A_t),(B_1,\dots,B_t)) = ((A_1g_1^{-1},\dots, A_{t}g_t^{-1}),(g_1B_1,\dots,g_tB_t))\]
Let $f \in I(V \cdot W)$, which now equals the left-hand side of
\eqref{eq:VW}. Define $h \in K[M_{\bk,\bl} \times M_{\bl,\bm}]$
by $h = \mu^\sharp(f)$, so that $h(A,B)=f(AB)$. Then $h$ is invariant with respect to the action of
$\lieg{GL}_{\bl}$, and moreover $h$ is in the ideal of $V \times W$. This latter fact implies that
\[ h \in I(V \times M_{\bl,\bm}) + I(M_{\bk,\bl} \times W);
\]
split $h=h_1+h_2$ accordingly. Applying the Reynolds operator
$\rho:K[M_{\bk,\bl} \times M_{\bl,\bm}] \to K[M_{\bk,\bl} \times
M_{\bl,\bm}]^{\lieg{GL}_\bl}$ yields $h=\rho(h_1) + \rho(h_2)$. By
$\lieg{GL}_\bl$-invariance of $V$ and $W$, $\rho(h_1)$ and $\rho(h_2)$
still are elements of $I(V \times M_{\bl,\bm})$ and $I(M_{\bk,\bl} \times W)$,
respectively.  Furthermore, $\rho(h_1)$ and $\rho(h_2)$ lie in
\[ K[M_{\bk,\bl} \times M_{\bl,\bm}]^{\lieg{GL}_\bl}
= K[M_{k_1,l_1} \times M_{l_1,m_1}]^{\lieg{GL}_{l_1}}
\otimes \cdots \otimes K[M_{k_t,l_t} \times
M_{l_t,m_t}]^{\lieg{GL}_{l_t}}. \]
By the First Fundamental Theorem for $\lieg{GL}_{l_i}$
applied to $k_i$ covectors and $m_i$ vectors the pullback of
multiplication $M_{k_i,l_i} \times M_{l_i,m_i} \to
M_{k_i,m_i}$ is a surjective homomorphism
\[ K[M_{k_i,m_i}] \to
K[M_{k_i,l_i} \times M_{l_i,m_i}]^{\lieg{GL}_{l_i}} \]
for all $i=1,\ldots,t$. Hence the pullback of multiplication $M_{\bk,\bl}
\times M_{\bl,\bm} \to M_{\bk,\bm}$ is a surjective homomorphism
\[ K[M_{\bk,\bm}] \to
K[M_{\bk,\bl} \times M_{\bl,\bm}]^{\lieg{GL}_{\bl}}; \]
let $\rho(h_1),\rho(h_2)$ lift under this surjection to
$\bar{h}_1,\bar{h}_2 \in K[M_{\bk,\bm}]$, respectively. Note that
$\bar{h}_1,\bar{h}_2$ are not unique if $l_i<\min\{k_i,m_i\}$ for some $i$,
but this is irrelevant here. We now have $\bar{h}_1 \in I(V \cdot
M_{\bl,\bm}), \bar{h}_2 \in I(M_{\bk,\bl} \cdot W)$. Moreover, restricted
to the image $M_{\bk,\bl}\cdot M_{\bl,\bm}$ we have $f = \bar{h}_1 +
\bar{h}_2$, since
\[
f(AB)=h(A,B)=\rho(h_1)(A,B)+\rho(h_2)(A,B)=\bar{h}_1(AB)+\bar{h}_2(AB)
\]
for all $A \in M_{\bk,\bl}, B \in M_{\bl,\bm}$. Hence $f' := f -
(\bar{h_1} + \bar{h}_2)$ vanishes on $M_{\bk,\bl}\cdot M_{\bl,\bm}$,
which contains both $V \cdot M_{\bl,\bm}$ and $M_{\bk,\bl}\cdot W$. But
then $f'$ lies in both $I(V \cdot M_{\bl,\bm})$ and $I(M_{\bk,\bl} \cdot
W)$, and hence $f \in I(V \cdot M_{\bl,\bm}) + I(M_{\bk,\bl} \cdot W)$
as claimed.
\end{proof}
We will use Theorem~\ref{thm:VW} to describe the ideal of $V\cdot W$ explicitly from the ideals of $V$ and $W$.
Before we can do this we need one more tool.
Suppose $\alpha \colon X \times K^n \to Y$ is a morphism
where $X$ and $Y$ are affine varieties over $K$, and
suppose $S$ is a closed subscheme of $Y$ defined by an
ideal $I$. Then there exists a uniquely determined
subscheme $S'$ of $X$ such that $\alpha(S' \times K^n) \subseteq
S$ and such that $S'$ is maximal with this property.
Scheme-theoretically, $S'$ is equal to $S' = \bigcap_{v \in
K^n} i_v^{-1}(\alpha^{-1}(S))$ where for $v \in K^n$, $i_v$
is the inclusion $X \to X \times \{v\} \subseteq X \times K^n$.
The ideal of $S'$ is determined as follows: Let $I'$ be the ideal in $K[X]$ generated by all functions of the form $f'_v = \alpha^\sharp(f)(\cdot,v) = f(\alpha(\cdot, v))$ where $f\in I$ and $v\in K^n$ is an arbitrary (closed) point. In other words, $f'_v(x) = f(\alpha(x,v))$ for $x\in X$.

Since $K[X\times K^n] = K[X]\otimes_K K[x_1,x_2,\dots,x_n]$, for any
$f\in K[Y]$, we may write $\alpha^\sharp(f)$ uniquely as
\begin{equation}\label{eq:alpha_sharp}
\alpha^\sharp(f) = \sum_i h_i \otimes m_i
\end{equation}
where the $m_i$ are some linearly independent monomials in
$K[x_1,x_2,\dots,x_n]$, and $h_i \in K[X]$.  It is now elementary to
check that $I'$ is generated by all $h_i\in K[X]$ that appear in such
an expression \eqref{eq:alpha_sharp} as $f$ runs through $I$.  Indeed,
the ideal generated by $f'_v$ where $v$ ranges over $K^n$ is precisely
the ideal generated by all $h_1,h_2,\dots,h_s$. This is easily seen
by picking $s$ points $v_1,v_2,\dots,v_s$ in $K^n$ such that $\det
[m_i(v_j)] \neq 0$, which is possible as $K$ is infinite and the $m_i$
are supposed to be linearly independent.  This observation is important
because it shows how to compute a finite list of generators for $I'$
out of finitely many generators for $I$: if $f_1,f_2,\dots,f_m$ generate
$I$, then the (finite) collection of all $h_i$s appearing in one of the
$\alpha^\sharp(f_j)$s generates $I'$.

We will apply this construction to the case where $X = M_{\bk,\bl}, K^n
= M_{\bl,\bm}$ and $\alpha = \mu$ equal to matrix multiplication. Then
for $V \subseteq M_{\bk,\bm}$, the ideal $I(V')$ is generated by all
functions on $M_{\bk,\bl}$ of the form $f(xB)$ where $B \in M_{\bl,\bm}$
is arbitrary, and $f \in I(V)$. In fact, thinking of the entries of $B$
as variables, we may expand $f(xB)$ as a polynomial in the entries of $B$;
the coefficients are then the required elements of the ideal of $V'$.
See also Example \ref{ex:Toric}.

The following corollary is crucial for explicit computations; it is
a slight generalisation of \cite[Lemma~12]{Allman04}, which gives the
same equations for $V \cdot M_{\bl,\bm}$. Although there the result
is stated only for subvarieties, their proof should also go through
essentially unchanged.

\begin{cor} \label{cor:VM}
Let $\bk,\bl,\bm \in \NN^t$ and let $V$ be a subscheme of $M_{\bk,\bl}$
with $V=V\cdot M_{\bl,\bl}$. Define the scheme $R_{\bl} \subseteq
M_{\bk,\bm}$ by the ideal generated by all $(l_i+1)$-minors of the $i$-th
component, for all $i=1,\ldots,t$. Then we have
\begin{equation} \label{eq:VM}
I(V \cdot M_{\bl,\bm})=I(V')+I(R_{\bl})
\end{equation}
where $V'$ is defined as in the preceding paragraph as the unique maximal
subscheme of $M_{\bk,\bm}$ such that $V'\cdot M_{\bm,\bl} \subseteq V$.
\end{cor}

It is well known that $M_{\bk,\bl}\cdot M_{\bl,\bm} = R_{\bl}$ as schemes,
and in particular that the ideal of $R_{\bl}$ is radical.

\begin{proof}
The inclusion $\supseteq$ follows from $V \cdot M_{\bl,\bm} \subseteq
V' \cap R_\bl$: First, $V \cdot M_{\bl,\bm} \subseteq R_{\bl} =
M_{\bk,\bl}\cdot M_{\bl,\bm}$ is clear. Second, $V \cdot M_{\bl,\bl} =
V$ implies that $V\cdot M_{\bl,\bm}\cdot M_{\bm,\bl} \subseteq V$, i.e.\
$V\cdot M_{\bl,\bm}\subseteq V'$.

For the opposite inclusion,
set $W:=M_{\bm,\bl} \cdot M_{\bl,\bm}$, and apply Theorem \ref{thm:VW} with
$(\bk,\bl,\bm,V,W)$ replaced by $(\bk,\bm,\bm,V',W)$.  Indeed,
$V' \cdot M_{\bm,\bl} \subseteq V$ by definition of $V'$, so that the left-hand
side of \eqref{eq:VM} is contained in $I(V' \cdot M_{\bm,\bl} \cdot
M_{\bl,\bm})$, which is the left-hand side of \eqref{eq:VW} with $(V,W)$
replaced by $(V',W)$. With this substitution the right-hand side of
\eqref{eq:VW} reads
\[ I(V' \cdot M_{\bm,\bm}) + I(M_{\bk,\bm} \cdot W) \]
which, as $M_{\bk,\bm} \cdot W=R_{\bl}$, equals the
right-hand side of $\eqref{eq:VM}$.
\end{proof}
The corollary is the reason why we had to use subschemes instead of
subvarieties: in general, $V'$ is not a variety even if $V$ is, so the
ideal $\sqrt{I(V')}$ of functions vanishing on the closed points of $V'$
may be larger than $I(V')$, and hard to compute. However, the corollary
shows that to compute $I(V \cdot M_{\bl,\bm})$ only the ideal $I(V')$
is needed, and for this ideal generators can be found as described above.

We will apply Theorem \ref{thm:VW} and its corollary in the following
setting: Let $V$ be a representation of $G$, $\Omega$ the set of all
irreducible characters, and for $\omega \in \Omega$ denote by $M_\omega$
a fixed irreducible representation of type $\omega$.  Then $V \cong
\bigoplus_{\omega\in \Omega} V[\omega]$, where $V[\omega]$ is the sum of
all submodules of $V$ isomorphic to $M_{\omega}$. Moreover, $V[\omega]
\cong M_{\omega} \otimes \Hom_G(M_\omega, V) \cong M_{\omega} \otimes
K^{m(\omega,V)}$ with $m(\omega,V)$ the multiplicity of $M_\omega$ in $V$.
In particular, if $W$ is another representation of $G$, then the space
of equivariant maps $V \to W$ is
\[ \Hom_G(V,W) \cong \bigoplus_{\omega \in \Omega} \Hom(K^{m(\omega,V)},K^{m(\omega,W)}).\]

The varieties we are interested in are subvarieties of $\Hom_G(V,W)$
stable by multiplication with $\End_G(V)$ or $\End_G(W)$ where $V,W$ are
some representations of $G$. So let $U,V,W$ be three representations of
$G$, and suppose $S \subseteq \Hom_G(U,V)$ and $T \subseteq \Hom_G(V,W)$
are subvarieties or subschemes. To apply Theorem~\ref{thm:VW}, we
may identify $\Hom_G(V,W)$ with $M_{\bk,\bl}$, $\Hom_G(U,V)$ with
$M_{\bl,\bm}$, and $\End_G(V)$ with $M_{\bl,\bl}$, by putting $l_i
= m(\omega_i,V)$,$k_i = m(\omega_i,W)$, and $m_i = m(\omega_i,U)$,
respectively, where $\Omega = \{\omega_1,\omega_2,\dots,\omega_t\}$.

With these identifications in place, the ideal of $T \cdot \End_G(V)
\cdot S$ is equal to $I(T \cdot M_{\bl,\bm}) + I(M_{\bk,\bl}\cdot
S)$. Similarly, if $T$ is stable by right-multiplication with
$\End_G(V,V)$, then the ideal of $T\cdot \Hom_G(U,V)$ may be computed
using Corollary~\ref{cor:VM}. We will see several applications of this
in the next section.

\section{Proofs of the main results} \label{sec:Proofs}

Before proving our main results, we investigate how $\Psi_T$ and
$\CVEM(T)$ behave under base changes. Thus let $T$ be a $G$-tree and write
$\GL_T$ for the product $\prod_{p \in \vertex(T)} \GL(V_p)^G$.  On the
one hand, this group acts on $\rep_G(T)$ by the action of $\GL(V_q)^G
\times \GL(V_p)^G$ on $(V_q \otimes V_p)^G$. We already encountered a
special case of this action in the proof of Proposition \ref{prop:Dense1}.
On the other hand, given $h \in \GL_T$ one can define a new $G$-spaced
tree $hT$ as follows: the underlying tree of $hT$ is the same as that of
$T$ and the space $V_p$ at each vertex is also the same as that of $T$,
but the bilinear form $(.|.)'_p$ is determined by
\[ (h_p u | h_p v)'_p:=(u | v)_p \text{ for } u,v \in V_p, \]
where the latter bilinear form is the one assigned to $p$ in $T$. Finally,
a vertex $p$ is based in $hT$ if and only if it is based in $T$, and
then the basis associated to it in $hT$ is $B_p':=hB_p$, where $B_p$
is the distinguished basis of $V_p$ in $T$. A representation $A$ of $T$
gives a representation of $hT$, also denoted $A$, by simply taking the
same tensors $A_{qp} \in V_q \otimes V_p$ along the edges.

\begin{lm} \label{lm:GTAction}
In the setting above we have $\Psi_{hT}(A)=h\Psi_T(h^{-1}A)$.
\end{lm}

\begin{proof}
If $T$ has only two vertices $p \sim q$, then this just the obvious
equality $A_{qp}=(h_q,h_p)(h_q^{-1},h_p^{-1})A_{qp}$. If $T$ has more than
two vertices, we pick any internal vertex $q$ of $T$ and split $T=*_i T_i$
and $A=*_i A_i$ at $q$. Assuming the result for all $T_i$ we find
\begin{align*}
\Psi_{hT}(A)
&= \sum_{b \in B_q'} \otimes_i (b | \Psi_{hT_i}(A_i))' \\
&= \sum_{b \in B_q'} \otimes_i (b | h\Psi_{T_i}(h^{-1}A_i))' \\
&= \sum_{b \in B_q} \otimes_i (h_q b | h\Psi_{T_i}(h^{-1}A_i))' \\
&= h \sum_{b \in B_q} \otimes_i (b | \Psi_{T_i}(h^{-1}A)) \\
&= h \Psi_{hT}(h^{-1} A).
\end{align*}
\end{proof}

In particular, this lemma implies that $\CVEM(T)=\CVEM(hT)$. For $G=\{1\}$
we note that if $T'$ is any spaced tree with the same underlying tree
as $T$ and the same $G$-modules $V_p$ at the vertices, but different
bilinear forms and different (orthonormal) bases, then there exists an
$h \in \GL_T$ with $hT=T'$. In this sense the variety $\CVGM(T)$ does
not depend on the chosen bases and forms, as long as they are compatible.

\begin{re} \label{re:ValencyTwo}
A stronger basis-independency holds at vertices of valency
two. There the operation $*$ boils down to matrix multiplication, or
composition of linear maps, and this will enable us to apply Theorem
\ref{thm:VW}. Indeed, let $U,V,W$ be vector spaces equipped with
non-degenerate symmetric bilinear forms and let $\Psi_1 \in W \otimes V$
and $\Psi_2 \in V \otimes U$ be arbitrary. Let $B$ be any orthonormal
basis of $V$. We claim that the element
\[ \sum_{b \in B} (b|\Psi_1) \otimes (b|\Psi_2) \in W \otimes U \]
does not depend on $B$, and under the identification $U \cong
U^*$ coincides with the linear map $U \rightarrow W$ which is the
composition of $\Psi_1$ and $\Psi_2$, considered as linear maps under the
identifications $W \otimes V=W \otimes V^*=\Hom(V,W)$ and $V \otimes U=V
\otimes U^*=\Hom(U,V)$. It suffices to verify this for rank-one tensors
$\Psi_1=w \otimes v$ and $\Psi_2=v' \otimes u$. The expression above
is then
\[ [\sum_{b \in B} (b|v)(b|v')] w \otimes u. \]
By the orthonormality of $B$ this reduces to $(v|v') w \otimes u$,
as claimed.

Thus, in hindsight, we could have left out the orthonormal bases at
vertices of valency $2$ in the definition of ($G$-)spaced trees, and
defined the operation $*$ as composition of linear maps. We have not
done so to keep the treatment of internal vertices uniform.
\end{re}

Next we observe that the map $\Psi_T$ defined in Section \ref{sec:Setup}
behaves well with respect to the group action. Let $T$ be a $G$-spaced
tree. Note that $G$ acts naturally on $\rep(T)$ by its action on each
tensor product $V_p \otimes V_q$ with $p \sim q$.

\begin{lm}
The map $\Psi_T \colon \rep(T) \to L(T)$ is $G$-equivariant.
\end{lm}

\begin{proof}
If $T$ has exactly two vertices the assertion is immediate. Otherwise,
let $q$ be an inner vertex of $T$ and split $T$ around $q$ into
$T_1,\ldots,T_k$. The $T_i$ are $G$-trees in a natural way. By
induction, we may assume that $\Psi_{T_i}$ is an equivariant map. Then
\begin{align*}
\Psi_T(gA) &= \sum_{b \in B_q} (b \mid \Psi_{T_1}(gA_1)) \otimes \dots
\otimes (b \mid \Psi_{T_k}(gA_k))\\
&= \sum_{b \in B_q} (b \mid g\Psi_{T_1}(A_1)) \otimes \dots \otimes
(b \mid g \Psi_{T_k}(A_k))\\
&= g \sum_{b \in B_q} (g^{-1} b \mid \Psi_{T_1}(A_1)) \otimes \dots \otimes
(g^{-1} b \mid g \Psi_{T_k}(A_k))\\
&= g \Psi_T(A),
\end{align*}
where the second equality follows from the $G$-invariance of $(. \mid
.)_q$ and the last equality follows from the fact that $g^{-1}$ permutes
$B_q$.
\end{proof}

This lemma implies that $\CVEM(T) \subseteq L(T)^G$. In what follows we
focus on the ideal of $\CVEM(T)$ inside $K[L(T)^G]$. To
obtain the ideal inside $K[L(T)]$, one just adds linear
equations cutting out $L(T)^G$ from $L(T)$.

\begin{re} \label{re:ValencyTwoB}
For $k=2$ the computation in the proof of the lemma can be replaced by the
following argument, using the notation of Remark \ref{re:ValencyTwo}. If
$U,V,W$ are $G$-modules and $V$ carries a $G$-invariant symmetric bilinear
form, then the unique bilinear map $W \otimes V \times V \otimes U
\rightarrow W \otimes U$ sending $(w \otimes v,v' \otimes u)$ to $(v|v')
w \otimes u$ is $G$-equivariant. So at vertices of valency $2$ it is
not crucial that $G$ permutes the basis.
\end{re}

With these preparations, we are now ready to prove our first main result.

\begin{proof}[Proof of Theorem \ref{thm:Procedure}]
Let $T$ be a $G$-spaced tree. We recursively express the ideal of
$\CVEM(T)$ into the ideals of $\CVEM(S)$ for substars $S$ of $T$ with
at least three leaves, as follows. First, if $T$ has only two vertices
$p \sim q$, then $\CVEM(T)=(V_p \otimes V_q)^G=L(T)^G$ and we are
done. Second, if $T$ is itself a star with at least three leaves, then
we are also done.  Third, suppose that $T$ contains a vertex $q$
of valency $2$, and split $T$ accordingly as $T = T_1 * T_2$, so that
\[ \CVEM(T)=\overline{ \{\Psi_1 * \Psi_2 \mid
	\Psi_i \in \CVEM(T_i) \text{ for } i=1,2 \}}. \]
Let $L_1$ be the space $\bigotimes_{p \in \leaf(T_1) \setminus \{ q\}} V_p$ and
$L_2 = \bigotimes_{p \in \leaf(T_2) \setminus q} V_p$.  Of course $L_1,L_2$ are
naturally $G$-representations. Now the map $(\Psi_1,\Psi_2) \mapsto
\Psi_1 * \Psi_2$ from $L(T_1)^G \times L(T_2)^G$ to $L(T)^G$ is just
matrix multiplication if we identify $L(T_1)^G$ with $\Hom_G(V_p,L_1)$
and $L(T_2)^G$ with $\Hom_G(L_2,V_p)$; see Remark
\ref{re:ValencyTwo}.

We want to apply Theorem \ref{thm:VW}. Recall the
definition of $\omega_i$ and $m(\omega_i,V)$ from
Section~\ref{sec:Multiplying}.  Now define $\bk,\bl,\bm$
by $k_i := m(\omega_i,L_1)$, $l_i  := m(\omega_i,V_q)$
and $m_i := m(\omega_i,L_2)$.  Then $\Hom_G(V_q,L_1) = M_{\bk,\bl}$,
$\Hom_G(L_2,V_q) = M_{\bl,\bm}$ and $\End_G(V_q) = M_{\bl,\bl}$.  Notice
that $V:= \CVGMG(T_1)$ (resp. $W := \CVGMG(T_2)$) are stable under right-
(resp. left-) multiplication with $M_{\bl,\bl}$, and $\CVGMG(T) = V\cdot
W = V \cdot M_{\bl,\bl} \cdot W$. Thus Theorem~\ref{thm:VW} applies and
we deduce that
\[ I(\CVGMG(T)) = I(\CVGMG(T_1) * \CVGMG(\flat_qT_2)) +
I(\CVGMG(\flat_q T_1)*\CVGMG(T_2),\]
where $\CVGMG(\flat_qT_i) \cong \Hom_G(L_i,V_q) \cong \Hom_G(V_q,L_i)
\cong (L_i \otimes V_q)^G$ because $G$ acts preserving the form.
Recursively, we may assume that the ideals of $\CVGMG(T_1)$ and $\CVGMG(T_2)$ have been computed.
Finally, the two terms on the right-hand side can be expressed into $I(\CVEM(T_1))$
and $I(\CVEM(T_2))$ using Corollary \ref{cor:VM}:
Following the recipe at the end of Section~\ref{sec:Multiplying}, we may compute e.g.\ $I(\CVGMG(T_1)*\CVGMG(\flat_q T_2))
= I(\CVGMG(T_1)\cdot \Hom_G(L_2,V_q))$.
This concludes the
case where $T$ contains a vertex of valency $2$.

Finally, if $T$ is neither a star nor contains a vertex of valency two,
then it contains an edge $p \sim r$ where both $p$ and $r$ are internal
vertices of valency at least three. Let $T'$ be the $G$-tree obtained
from $T$ by inserting two vertices $q_1$ and $q_2$ between $p$ and $r$ so
that $p \sim q_1 \sim q_2 \sim r$, setting $V_{q_1}:=V_r$ with the same
bilinear form and basis, and $V_{q_2}:=V_p$ with the same bilinear form
and basis. Note that every $G$-spaced substar of $T'$ with at least three
leaves is also a $G$-spaced substar of $T$. This is why we inserted two
vertices rather than one: what space should we attach to a single vertex
between $p$ and $q$? See below for a comment on this. By the previous
construction, we can express the ideal of $\CVEM(T')$ in the ideals
of $\CVEM(S)$ of all substars $S$ of $T'$, hence of $T$, with at least
three leaves. So we are done if can show that $\CVEM(T)=\CVEM(T')$. But
any $A' \in \rep_G(T')$ gives rise to an $A \in \rep_G(T)$ by setting
$A_{pr}:=A'_{pq_1} A'_{q_1q_2} A'_{q_2r}$. Using Remark
\ref{re:ValencyTwo}
one finds that $\Psi_T(A)=\Psi_{T'}(A')$. Conversely, for any $A \in
\rep_G(T)$ we can factorise $A_{pr}$ into $A'_{pq_1}A'_{q_1q_2}A'_{q_2r}$
with $G$-invariant factors. This gives a representation $A'$ of $T'$
with $\Psi_{T'}(A')=\Psi_T(A)$. This concludes the proof of the theorem.
\end{proof}

\begin{re}
Note that for $G=1$, i.e., for the general Markov model, the proof above
can be simplified slightly: one does not need the decomposition into
isotypic components, and may apply Theorem \ref{thm:VW} with
$t=1$ directly.
\end{re}

The proof above yields to the following algorithm for computing
$I(\CVEM(T))$ from the ideals of substars.

\begin{alg}\label{alg:Procedure}\
\begin{description}
\item[Input] a $G$-spaced tree $T$ and finite generating sets of the
ideals $I(\CVEM(S))\subseteq K[L(S)^G]$ for all substars $S$ in $T$
with at least three leaves.
\item[Output] a finite generating set of the ideal $I(\CVEM(T))\subseteq
K[L(T)^G]$.
\item[Procedure]\
\begin{enumerate}
\item If $T$ contains only two vertices $p \sim q$, then
return the empty set and quit.
\item If $T$ is a star with at least three leaves, then a finite
generating set of $I(\CVEM(T))$ is part of the input;
return this set and quit.
\item If $T$ has a vertex of valency $2$, then choose such a vertex $q$
and split $T=T_1 * T_2$ at $q$. Apply this algorithm to $T_1$ and to
$T_2$ (with the ideals of their substars) to find finite generating sets
$F_1,F_2$ of the ideals of $\CVEM(T_1)$ and $\CVEM(T_2)$, respectively.
Let $L_1,L_2$ be as in the proof of theorem \ref{thm:Procedure},
identify $L(T)^G=\Hom_G(L_1,L_2)$ with $M_{\bk,\bm}$ as in that proof, and
write $\bl$ for the tuple of multiplicities $m(\omega_i,V_q)$.  Identify
$L(T_1)$ with $M_{\bk,\bl}$ and $L(T_2)$ with $M_{\bl,\bm}$, so that $F_1
\subseteq K[M_{\bk,\bl}]$ and $F_2 \subseteq K[M_{\bl,\bm}]$. Write $\Psi$
for an element in $M_{\bk,\bm}$ whose coordinates are variables.
\begin{enumerate}
\item Let $F'$ denote the collection of all $(l_i+1) \times
(l_i+1)$-minors of the $i$-th component of $\Psi$ for all
$i=1,\ldots,t$.
\item For an element $\Psi_0 \in M_{\bm,\bl}$ with new
variables as coordinates, expand $f(\Psi \cdot \Psi_0)$ for
each $f \in F_1$, and take all coefficients of monomials in
$\Psi_0$, which are polynomials in $\Psi$. Collect these
polynomials in $F_1'$.
\item For an element $\Psi_0 \in M_{\bk,\bl}$ with new
variables as coordinates, expand $f(\Psi_0 \cdot \Psi)$ for
each $f \in F_2$, and take all coefficients of monomials in
$\Psi_0$, which are polynomials in $\Psi$. Collect these
polynomials in $F_2'$.
\end{enumerate}
Return $F' \cup F_1' \cup F_2'$ and quit.
\item Take an edge $p \sim r$ in $T$ connecting two vertices
of valency at least three. Construct a $G$-tree $T'$ by
inserting two new vertices $q_1,q_2$ such that $p \sim
q_1 \sim q_2 \sim r$ and setting $V_{q_1}:=V_r$ with the
same basis and bilinear form and $V_{q_2}:=V_p$ with the
same basis and bilinear form. Run this algorithm on $T'$,
return the same output as for $T'$, and quit.
\end{enumerate}
\end{description}
\end{alg}

Although the tree grows in the last step, it is easy to see that this
algorithm terminates: after inserting vertices, in the call with $T'$
the tree is broken into two trees, each of which have strictly less
substars with at least three vertices. This algorithm is
partly carried out in Example \ref{ex:Toric}.

\begin{re}\label{re:ValencyTwoC}
In the last step of both the proof and the algorithm we
could also have inserted a single vertex $q$ between $p$ and
$r$, with $V_q$ equal to the $G$-module having multiplicities
$m(\omega_i,V_q)=\min\{m(\omega_i,V_p),m(\omega_i,V_r)\}$ for all
$i$, so that all $G$-equivariant maps $V_r \rightarrow V_p$ factorise
through $V_q$. One can show that this $V_q$ carries a $G$-invariant,
non-degenerate symmetric bilinear form since $V_p$ and $V_r$ do. This
set-up would have raised two minor problems. First, the object $T'$
thus constructed is strictly speaking not a $G$-tree, as $V_q$ may
not have an orthogonal basis permuted by $G$. But as we saw in Remarks
\ref{re:ValencyTwo} and \ref{re:ValencyTwoB} this is not really a problem:
we can still apply Theorem \ref{thm:VW} at $q$ to split $T'$ into smaller
trees. Second, the $G$-spaced stars $S_p',S_r'$ with centres $p$ and $r$
in $T'$ are not equal to the $G$-spaced stars $S_p,S_r$ around $p$ and $r$
in $T$. Hence after expressing $I(\CVEM(T'))$ in the ideals $I(\CVEM(S))$
for all stars $S$ with at least three leaves in $T'$, we still need to
express the ideals of $\CVEM(S'_p)$ and $\CVEM(S'_r)$ in $\CVEM(S_p)$ and
$\CVEM(S_r)$, respectively, to prove the theorem. The following lemma
does just that. This would give a slight variant of the algorithm above.
\end{re}

\begin{lm}\label{le:PushFwd}
Let $S',S$ be $G$-spaced stars with the same underlying star having $q$
as centre and $p_1,\ldots,p_k$ as leaves. Suppose that both stars have
the same space $V_q$ with the same basis $B_q$ and that we are given
$G$-equivariant injections $\tau_i:V_{p_i}' \rightarrow V_{p_i}$ for
$i=1,\ldots,k$, where $V_{p_i}'$ and $V_{p_i}$ are attached to $p_i$
in $S'$ and $S$, respectively. Denote by $\tau$ the induced injection
$L(S') \rightarrow L(S)$. Then
\[ I(\CVEM(S'))=\tau^\sharp I(\CVEM(S)). \]
In particular, a finite generating set for $I(\CVEM(S))$ gives a finite
generating set for $I(\CVEM(S'))$ under pull-back by $\tau$.
\end{lm}

\begin{proof}
For the inclusion $\supseteq$ note that any $G$-representation
$A=(A_{p_i,q})_i$ of $S$ gives rise to a representation $A'=(\tau_i
A_{p_i,q})_i$ of $S'$ satisfying $\tau\Psi_S(A)=\Psi_{S'}(A')$. Hence $\tau$
maps $\CVEM(S)$ into $\CVEM(S')$.

For the inclusion $\subseteq$ note that, as $V_{p_i}$ is a completely
reducible $G$-module, there exist $G$-equivariant surjections
$\pi_i:V_{p_i}' \rightarrow V_{p_i}$ with $\pi_i \tau_i=\id_{V_{p_i}}$.
Now the induced projection $\pi:L(S')\rightarrow L(S)$ maps $\CVEM(S')$ into
$\CVEM(S)$, and if $f \in I(\CVEM(S))$, then
$f=\tau^\sharp(\pi^\sharp f)$, where $\pi^\sharp f$ lies in $I(\CVEM(S'))$.
\end{proof}

Now we prove our second main result.

\begin{proof}[Proof of Theorem \ref{thm:Flattening}.]
Recall the statement of the theorem: for any $G$-tree $T$ we have
\[ \tag{*} \label{eq:VMG} I(\CVGMG(T)) =
	\sum_{q \in \vertex(T)}I(\CVGMG(\flat_qT)). \]
We proceed by induction. First, the statement is a tautology for a star
$T$. Next, suppose that $T$ has an inner vertex $q$ of valency $2$ and
split $T=T_1*T_2$ at $q$.  By induction we may assume that the theorem
holds for $T_i$. The proof of Theorem \ref{thm:Procedure} shows that
\[ I(\CVEM(T))=I(\CVEM(T_1)*\CVEM(\flat_q T_2)) +
	I(\CVEM(\flat_q T_1)*\CVEM(T_2)), \]
so it suffices to prove that each of these terms is contained
in the right-hand side of \eqref{eq:VMG}; we do so for the first
term. We use the notation $L_1,L_2,\bk,\bl,\bm$ from the proof of
Theorem \ref{thm:Procedure}.  By Corollary \ref{cor:VM} the ideal of
$\CVEM(T_1)*\CVEM(\flat_q T_2)$ is spanned by polynomials in $\Psi \in
L(T)=\Hom_G(L_2,L_1)$ of the following two forms:
\begin{enumerate}
\item for all $i=1,\ldots,t$ the $(l_i+1)$-minors of the $i$-th component
of $\Psi$, regarded as an element of $M_{\bk,\bm}$; and
\item all polynomials of the form $\Psi \mapsto f(\Psi\Psi_0)$, where
$f \in I(\CVEM(T_1))$ and $\Psi_0$ is some element of $\Hom_G(V_p,L_2)$.
\end{enumerate}
The first type of elements are clearly equations for $\CVGM(\flat_q T)$,
so we need only worry about the second type of equations. By induction we
may assume that $f$ is an equation for $\CVGM(\flat_{r} T_1)$ for some
vertex $r$ of $T_1$. But then $\Psi \mapsto f(\Psi\Psi_0)$ vanishes on
$\CVGM(\flat_r T)$, and we are done.

Finally, if $T$ is not a star and does not contain a vertex of valency
$2$, then we proceed as in the proof of Theorem~\ref{thm:Procedure}.
We choose an edge $p \sim r$ in $T$ where both $p$ and $r$ have valency
at least three, and insert vertices $q_1,q_2$ with $p \sim q_1 \sim
q_2 \sim r$ and $V_{q_1}:=V_r$ and $V_{q_2}:=V_p$ to obtain a new
$G$-spaced tree $T'$. We claim that both sides of \eqref{eq:VMG}
remain unchanged in replacing $T$ by $T'$. For the left-hand side
this was proved in the proof of Theorem~\ref{thm:Procedure}. The
right-hand side gains two terms, namely, $I(\CVEM(\flat_{q_1} T'))$ and
$I(\CVEM(\flat_{q_2} T'))$.  However, the definition of flattening readily
implies that $\CVEM(\flat_p T) \subseteq \CVEM(\flat_{q_2} T')$, so that
$I(\CVEM(\flat_{q_2} T')) \subseteq I(\CVEM(\flat_p T))$.  Similarly,
we find $I(\CVEM(\flat_{q_1}T')) \subseteq I(\CVEM(\flat_r T))$, and
hence the two extra terms on the right-hand side of \eqref{eq:VMG}
are redundant. Now the theorem for $T$ follows from that for $T'$,
which in turn follows by induction as in the previous case.
\end{proof}

Finally, we prove the last statement of Proposition \ref{prop:Dense2}
which says that $\overline{\CVEM(T')V_{r'}}$ can be computed from the
ideal of $\CVEM(T')$.

\begin{proof}[Proof of Proposition \ref{prop:Dense2}.]
We will apply Corollary \ref{cor:VM} where $\CVEM(T')$ will play the
role of $V$ and $V_{r'}$ will play the role of $M_{\bl,\bm}$. To this
end we proceed as in the proof of Theorem \ref{thm:Procedure}. First
set $L:=\bigotimes_{p \in \leaf(T')\setminus\{r'\}}V_p$ and let
$\omega_1,\ldots,\omega_t$ be the characters of $G$.  Next define
$\bk,\bl,\bm$ as follows: $k_i:=m(\omega_i,L)$, $l_i:=m(\omega_i,V_{r'})$,
and $m_i:=m(\omega_i,KG)$ with $KG$ the left regular representation
of $G$. View $\CVEM(T')$ as a subvariety of $\Hom_G(V_{r'},L)$, which
can be thought of as $M_{\bk,\bl}$, and view $V_{r'}$ as the space
$\Hom_G(KG,V_{r'})$ under the isomorphism $v \mapsto (g \mapsto gv)$,
which can be thought of as $M_{\bl,\bm}$.  It is easy to see that
$\CVEM(T')$ is closed under composition with $\Hom_G(V_{r'},V_{r'})$,
which is $M_{\bl,\bl}$, so that Corollary \ref{cor:VM} applies. We
conclude that $I(\CVEM(T')V_{r'})$, regarded as a subset of $M_{\bk,\bm}$
is generated by the rank-$(l_i+1)$-minors of the $i$-th block for
$i=1,\ldots,t$ and the polynomials $L \rightarrow K$ of the form $\Psi
\mapsto f(\Psi \Psi_0)$, where $\Psi$ is regarded as a $G$-homomorphism
$KG \rightarrow L$, $f$ runs over $I(\CVEM(T'))$ and $\Psi_0$ over all
elements of $\Hom_G(V_r,KG)$, which is $M_{\bl,\bm}$.

Of course, like in Algorithm \ref{alg:Procedure}, this can be made
into a finite set of generators by taking the entries of $\Psi_0$ to
be variables, taking $f$ in a finite generating set of $I(\CVEM(T'))$,
expanding, and taking the coefficients of the monomials in $\Psi_0$.
\end{proof}

We have now reduced the ideals of our equivariant models to those for
stars, and argued their relevance for statistical applications. The
main missing ingredients for successful applications are equations
for star models. These are very hard to come by: \cite{Garcia05}
posed several conjectures concerning these for the general
Markov model, and special cases of these conjectures were proved in
\cite{Allman04,Landsberg04,Landsberg06}. For certain important equivariant
models equations were found in \cite{Casanellas05, Sturmfels05b}. Roughly
speaking, the less symmetry one imposes on the model, the harder it is
to find equations. The following proposition offers some explanation
for this.

\begin{prop} \label{prop:GeneralStar}
Suppose that $T$ is a star with (based) centre $r$. Let $b_1,\ldots,b_s$
be representatives of the $G$-orbits on $B_r$ and denote by $G_i$ the
stabiliser of $b_i$ in $G$. Let $C_i$ denote the cone of pure tensors
in $\bigotimes_{p \in \leaf(T)} (V_p^{G_i}) \subseteq L(T)$, and denote
by $\rho$ the Reynolds operator for $G$. Then
\[ \CVEM(T)=\overline{\rho(C_1)}+\ldots+\overline{\rho(C_s)}, \]
where the addition corresponds to taking the join of these varieties.
\end{prop}

\begin{proof}
Elements in an open dense subset of $\CVEM(T)$ look like
\begin{align*}
\sum_{b \in B_r} \bigotimes_{p \in \leaf(T)} A_{pr} b
&=\sum_{i=1}^s \sum_{g \in G/G_i} \bigotimes_{p \in
\leaf(T)} A_{pr} gb_i\\
&=\sum_{i=1}^s \sum_{g \in G/G_i} g\left(\bigotimes_{p \in
\leaf(T)} A_{pr} b_i\right)\\
&=\sum_{i=1}^s |G/G_i| \rho\left(\bigotimes_{p \in \leaf(T)}
A_{pr} b_i\right)\\
&=\sum_{i=1}^s \rho\left(|G/G_i| \bigotimes_{p \in \leaf(T)}
v_{i,p}\right),
\end{align*}
where $v_{ip}=A_{pr} b_i \in V_p^{G_i}$; the latter element clearly lies
in the join $\overline{\rho(C_1)}+\ldots+\overline{\rho(C_s)}$. This
argument can be reversed to show the opposite inclusion.
\end{proof}
This is particularly interesting in the case when for all internal vertices $q$, $B_q$ is a single $G$-orbit. As usual, we may assume that $T$ is a star, and then $\CVEM(T)$ is simply $\overline{\rho(C)}$ where $C$ is the set of pure tensors in $\bigotimes_{p \in \leaf(T)} (V_p^{H})$ with $H = G_b$ the stabiliser of some element $b \in B_q$. $\rho$ being a linear projection now means that the ideal may be computed by elimination theory, at least in principle. This applies to (1), (2), and (6) in Example~\ref{ex:Models}.

In the following section we record some further observations for abelian
groups $G$.

\section{Abelian groups and toricness} \label{sec:Toric}

In this section we collect some results on the equivariant model for an
abelian group $G$. The fact that all irreducible representations of $G$
are one-dimensional makes $G$-equivariant models somewhat easier to
analyse than general equivariant models. Recall that an element $v$
in a $G$-representation is called a weight vector if it is a common
eigenvector of all elements of $G$; in that case the function $\lambda:G
\rightarrow K^*$ determined by $gv=\lambda(g)v$ is a character of $G$. We
also say that $G$ scales $v$ by $\lambda$. The following results are a
slight generalisation of results in
\cite{Evans93,Sturmfels05b}.

\begin{prop} \label{prop:Star}
Suppose that $T$ is a star with (based) centre $r$, that $G$ is abelian
and that $B_r$ is a single $G$-orbit. Then $\CVEM(T)$ is a toric variety.
More specifically, there exist tori $S_p$ in $\lieg{GL}(V_p)$ for $p \in
\leaf(T)$, diagonalised by certain bases of $G$-weight vectors such that
$\prod_{p \in \leaf(T)} S_p$ stabilises $\CVEM(T)$ with a dense orbit.
\end{prop}

\begin{proof}
Fix $b \in B_r$. A typical element of $\CVEM(T)$ looks like
\[ \sum_{g \in G/G_b} g \left (\bigotimes_{p \in \leaf(T)} v_p\right ) \]
with $v_p \in V_p^{G_b}=:V_p'$. As $G$ is abelian, $V_p'$ is a
$G$-module. Choose any basis of $V_p'$ that diagonalises $G$ and let $S_p$
be the associated torus in $\lieg{GL}(V_p')$, regarded as a torus in
$\lieg{GL}(V_p)$ acting trivially on a $G$-stable complement of $V_p'$
in $V_p$. Set $S:=\prod_{p \in \leaf(T)} S_p$. Then $(s_p)_p \in S$
sends the element above to
\[ \sum_{g \in G/G_b} g \left(\bigotimes_p s_p v_p\right), \]
which again lies in $\CVEM(T)$. Moreover, as each $S_p$ has a dense
orbit on $V_p'$, $S$ has a dense orbit on $\CVEM(T)$.
\end{proof}

\begin{re}
If $G$ is abelian and has $k$ orbits on $V_p$, then
$\CVEM(T)$ for a star $T$ is a join of $k$ toric varieties.
This fact is exploited in \cite{Casanellas05} for the strand-symmetric
model, where $k=2$.
\end{re}

\begin{thm} \label{thm:Toric}
Suppose that $G$ is an abelian group and that $T$ is a $G$-tree in which
$G$ has a single orbit on all $B_p$ with $p \in \internal(T)$. Then
$\CVEM(T)$ is a toric variety. More precisely, there exists a torus $S$
acting linearly on $L(T)$, stabilising $\CVEM(T)$ with a dense orbit,
and diagonalised by a tensor product of $G$-weight bases of the $V_p,
p \in \leaf(T)$.
\end{thm}

There is a subtlety here: unlike in Proposition \ref{prop:Star} such
a torus can in general {\em not} be found in $\prod_{p \in \leaf(T)}
Z_{\lieg{GL}(V_p)} G$.

\begin{proof}
We proceed by induction. First, if $T$ is the single edge $pq$,
then $\CVEM(T) = (V_p \otimes V_q)^G = \bigoplus_{\lambda + \mu = 0}
V_p[\lambda] \otimes V_q[\mu]$ where $(\lambda,\mu)$ ranges over
pairs of characters of $G$. Let $S$ be a maximal torus in $\lieg{GL}((V_p
\otimes V_q)^G)$ which is diagonal with respect to a basis of pure
tensors $v_\lambda \otimes v_{-\lambda}$ with $v_\lambda \in V_p$ and
$v_{-\lambda} \in V_q$ weight vectors of weights $\lambda,-\lambda$. View
$S$ as a torus in $\lieg{GL}(V_p \otimes V_q)$ acting trivially on all
$V_p[\lambda] \otimes V_q[\mu]$ with $\lambda+\mu \neq 0$. This $S$
has the properties claimed in the theorem.

Second, if $T$ is a star, then the proposition above does the trick.
Third, if $T$ is neither a star nor an edge, then let $q \in \internal(T)$
be any internal vertex of valency two. As in the proof of Theorem
\ref{thm:Procedure} we may add such a vertex, if necessary, without
changing $\CVEM(T)$---and in fact, if $q$ is inserted between the
internal vertices $p,r$, then, as $B_p,B_r$ are $G$-orbits, $V_q:=KG$ is
sufficiently large.  Write $T=T_1 * T_2$ at $q$, let $L_i := \bigoplus_{p
\in \leaf(T_i) \setminus q} V_p$, and let $S_1,S_2$ be the tori whose existence
is claimed by the theorem for the $G$-trees $T_i$. In particular, $S_i$
is diagonalised by the tensor product of a $G$-weight basis $C_i$ of $L_i$
and a $G$-weight basis $D_i$ of $V_q$ (such a basis is always orthogonal and may be chosen orthonormal). But since $V_q$ is a permutation
module with a single orbit, every weight occurs at most once in $V_q$,
so that (after scaling) $D_1=D_2=:D$. Now we let $S_i$ act on $L_i$
as follows: let $c_i \in C_i$ have $G$-weight $\lambda$. If $-\lambda$
is not a weight in $V_q$, then $S_i c_i:=c_i$. If, on the other hand,
$-\lambda$ is a weight in $V_q$, then it is the weight of a unique $d
\in D$, and we let $S_i$ scale $c_i$ by the character with which it
scales $d \otimes c_i \in L(T_i)$. Now $S:=S_1 \times S_2$ acts on $L_1
\otimes L_2 = L(T)$, and is diagonalised by the tensor product of $C_1$
and $C_2$.

Finally we verify that $S$ stabilises $\CVEM(T)$ with a dense
orbit.  To see this, observe that the map $(V_q \otimes L_1)^G \times
(V_q \otimes L_2)^G \to (L_1 \otimes L_2)^G$ given by $(\Psi_1, \Psi_2)
\mapsto \sum_{b \in B_q} (\Psi_1 \mid b) \otimes (\Psi_2 \mid b)$ is $S_1
\times S_2$-equivariant: it sends $(d_1 \otimes c_1,d_2 \otimes c_2)$,
where $d_i$ and $c_i$ have opposite weight, to $(d_1 \mid d_2)(c_1
\otimes c_2)$, which scales with the same $S$-characters by definition
of the action of $S$. Since $S_i$ has a dense orbit on
$\CVEM(T_i)$, $S$ has a dense orbit on $\CVEM(T)$.
\end{proof}

Theorem \ref{thm:Toric} reduces the computation of the ideals of certain
equivariant models to the combinatorics of toric varieties (where we do not require toric varieties to be normal).  However, this combinatorics can be very intricate, and it requires great ingenuity
to find explicit generators as in \cite{Sturmfels05b}.
We conclude with an example.

\begin{ex} \label{ex:Toric}
First let $T$ be a star with centre $q$ and four leaves
$p_1,\ldots,p_4$. Let $G =\ZZ_2 = \{1,x\}$ and $V_p = KG$ for all vertices
$p \in T$, with basis $G$ and induced form. For this $G$ and $V_p$,
a star with four leaves is the smallest $G$-star for which $\CVGMG(T)
\subsetneq L(T)^G$.

We are free to choose any basis on $L(T)$, so we pick the orthonormal
product basis of the basis of $KG$ diagonalising $G$. Let us denote
this basis by $(t,s)$ where $t=1+x$ spans the trivial, and $s=1-x$
the sign representation in $KG$.  We will label the basis of $L(T)$
given by the pure tensors $b_1 \otimes b_2 \otimes b_3 \otimes b_4$
($b_i \in \{ t,s\}$) as follows: for a subset $I \subseteq \{1,2,3,4\}$
let $b_I = b_1 \otimes b_2 \otimes b_3 \otimes b_4$ where $b_i= s$ if
$i \in I$, and $b_i = t$ otherwise.  Then $L(T)^G$ is spanned by all $b_I$
such that $\vert I \vert$ is even.

As $\Hom_G(KG,KG) \cong K^2$, an equivariant representation of $T$
is specified by $8$ parameters $y_i,x_i$ ($i = 1,2,3,4$) where $x_i$
is dual to $s$ and $y_i$ is dual to $t$ in $KG$, and $A_{p_iq} = (x_is +
y_it) \otimes 1 + (y_it - x_is) \otimes x$ (where $p_1,p_2,p_3,p_4$ are
the leaves).  By Proposition~\ref{prop:GeneralStar}, $\CVGMG(T) = \rho(P)$
where $P$ is the variety of pure tensors in $L(T)$. Specifically, if $A =
(x_i,y_i)_{i=1}^{4}$ is a representation, then
\[ \Psi_T(A) = 2\rho\Bigl(\bigotimes_i (x_is + y_it)\Bigr). \]
Let the variables dual to the basis $(b_I)_{ I \subseteq \{ 1,2,3,4\}}$
be denoted by $x_I$.  Then $x_I(\Psi_T(A)) = 0$ if $\vert I \vert$
is odd, and
\[ x_I(\Psi_T(A)) = \Biggl (\prod_{i\in I}x_i \Biggr ) \Biggl (\prod_{i \not \in I} y_i \Biggr ), \]
if $\vert I \vert$ is even. Thinking of $\CVGMG(T)$ as a subvariety of
$L(T)^G$, for determining the ideal we need to consider only functions in
the $x_I$ with $\vert I \vert$ even.  There are some obvious relations,
namely, if $I,J$ are two subsets of $\{1,2,3,4\}$ with even number of
elements and with complements $I^c$, $J^c$, respectively, then
\[ f_{I,J} := x_{I}x_{I^c} - x_Jx_{J^c} \]
is in the ideal of $\CVGMG(T)$.  Note that it is enough to consider
$f_{I,J}$ where every subset appears once as $I$ or $J$, because $f_{I,J}
+ f_{J,J'} = f_{I,J'}$, and if $\vert I\vert = \vert J \vert = 2$, then
$f_{I,J}$ is nonzero only if $I \cap J$ contains one element. Altogether
it follows that the ideal generated by the $f_{I,J}$'s is already
generated by
\[ f_{\emptyset, \{ 1,2\}},  f_{\emptyset,\{1,3\}}, f_{\emptyset,\{1,4\}}.\]
One can show that $I(\CVGMG(T))$ is generated by the $f_{I,J}$'s. Indeed,
by \cite[Theorem 3.1]{Diaconis98} this boils down to showing that these
$f_{I,J}$'s correspond to a Markov basis for the module of $\ZZ$-linear
relations among the $8$ vectors $(\ba,\bb) \in \{0,1\}^4 \times \{0,1\}^4
\subseteq \NN^4 \times \NN^4$ where $|\ba|:=\sum_i a_i$ is even and
$\bb=\one-\ba$ with $\one=(1,1,1,1)$. We omit the combinatorial details
here.

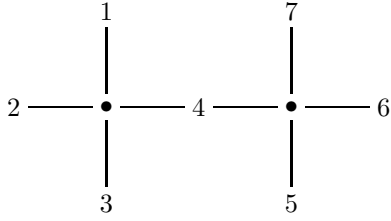
\begin{figure}[h]
\[
\xymatrix{ & 1 \ar@{-}[d]& & 7 \ar@{-}[d]& \\
           2 \ar@{-}[r]& \bullet \ar@{-}[r]& 4 \ar@{-}[r] & \bullet \ar@{-}[r]& 6 \\
           & 3\ar@{-}[u] & & 5 \ar@{-}[u] & }
\]
\caption{\label{fig:Tree}The spaced tree $T$.}
\end{figure}

To illustrate Theorem~\ref{thm:Procedure}, we now consider a $G$-tree $T$
obtained by gluing together two stars as above at one common leaf (see
Figure~\ref{fig:Tree}). Notice that by Remark~\ref{re:ValencyTwoC} and Lemma~\ref{le:PushFwd} $\CVGMG(T) = \CVGMG(T')$ where $T'$
is the tree $T$ with vertex $4$ removed and the centres of the two stars
in $T$ joined by an edge.

Algorithm \ref{alg:Procedure} first identifies a vertex of valency $2$;
here vertex $4$. We then write $T = T_1*T_2$, with $T_1$ the left and
$T_2$ the right star with four leaves each. By the above we know the
ideals of $\CVGMG(T_i)$. The content of Theorem~\ref{thm:Procedure}
in this situation is that
\[ I(\CVGMG(T)) = I(\CVGMG(T_1)*L(T_2)^G) + I(L(T_1)^G*\CVGMG(T_2)). \]
Because of the symmetry of the problem, we only consider the first
summand.  We keep the notation introduced above with respect to $T_i$:
the variables on $L(T_1)^G$ will be $x_I$ ($I \subseteq \{ 1,2,3,4\}$,
$\vert I\vert$ even) and those on $L(T_2)^G$ will be $y_I$ ($I \subseteq
\{ 4,5,6,7 \}$, $\vert I \vert$ even). Finally the variables on $L(T)^G$
will be $z_J$ where $J \subseteq \{ 1,2,3,5,6,7 \}$  has an even number
of elements; $J$ corresponds to the basis vector $b_J = b_1 \otimes
b_2 \otimes b_3 \otimes b_5 \otimes b_6 \otimes b_7$ where $b_i = s$
if $i \in J$ and $b_i = t$ otherwise. We also adopt the convention that
$x_I,y_I,z_I = 0$ if $\vert I \vert$ is odd.

The ideal of $\CVGMG(T_1)*L(T_2)^G$ is generated by $I':=I(\CVEM(T_1)')$
and certain $2\times 2$-minors, since both representations of $\ZZ_2$
occur with multiplicity one in $V_4$. These minors are of the form
\[ z_{I_1 \cup I_2} z_{I'_1 \cup I'_2}-z_{I_1\cup I'_2}
z_{I_1' \cup I_2} \]
where $I_1,I'_1 \subseteq \{1,2,3\},I_2,I'_2 \subseteq \{5,6,7\}$ are
all distinct and either all even or all odd. Next we show how to find
generators of $I'$. The space $L(T_2)^G$ is isomorphic to $\Hom_G(V_4,
\bigotimes_{ p \in \leaf(T_2) \setminus \{ 4\}}) \cong M_{\bm,\bl}$
with $\bm=(4,4)$ and $\bl=(1,1)$. Similarly,
\[ L(T)^G \cong \Hom_G\Biggl(\bigotimes_{p \in \leaf(T_2)
\setminus \{ 4\}}V_p,
\bigotimes_{p \in \leaf(T_1) \setminus \{ 4 \} }V_p\Biggr) \cong M_{\bk,\bm}\]
with $\bk=(4,4)$ and $\bm$ as above.

Let $\Psi_0 \in L(T_2)^G$ be arbitrary and let $\Psi \in L(T)^G$. Then
$\Psi \Psi_0 \in \Hom_G(V_4,\bigotimes_{p \in \leaf(T_1)\setminus \{
4\}}V_p) = L(T_1)^G$ and a straightforward computation shows that
\[ x_{I}(\Psi \Psi_0) =
\begin{cases}
	\sum_{J \subseteq \{5,6,7\} }
	z_{J \cup I}(\Psi)y_{J}(\Psi_0) &
		\text{ if $4 \not \in I$, and } \\
	\sum_{J \subseteq \{5,6,7\} }
	z_{J \cup I\setminus \{4\}}(\Psi)y_{J \cup \{ 4\}}(\Psi_0) &
\text{ if } 4 \in I.
\end{cases} \]
To avoid clumsy notation, let us write $\tilde x_I$ for the function $x_I(\Psi \Psi_0)$ in the arguments $(\Psi,\Psi_0)$.
Then for example
\begin{align}\label{eq:relation1}\tilde x_{\emptyset} = & z_{\emptyset}y_{\emptyset} + z_{\{5,6\}}y_{\{5,6\}} + z_{\{5,7\}}y_{\{5,7\}} + z_{ \{ 6,7\}}y_{\{6,7\}}\\
 \tilde x_{\{1,2,3,4\}} = & z_{\{1,2,3,5\}}y_{\{4,5\}} + z_{\{1,2,3,6\}}y_{\{4,6\}} + z_{\{1,2,3,7\}}y_{\{4,7\}}\\*
 \notag &+ z_{ \{ 1,2,3,5,6,7\}}y_{\{4,5,6,7\}}\\
 \tilde x_{\{1,2\} } = & z_{\{1,2\}}y_\emptyset + z_{\{ 1,2,5,6 \}}y_{\{5,6\}} + z_{\{1,2,5,7\}}y_{\{5,7\}} + z_{\{1,2,6,7\}}y_{\{6,7\}}\\
 \intertext{and finally}
\label{eq:relation4} \tilde x_{\{3,4\}} = & z_{\{3,5\}}y_{\{4,5\}} + z_{\{3,6\}}y_{\{4,6\}} + z_{\{3,7\}}y_{\{4,7\}}
 +z_{\{3,5,6,7\}}y_{\{4,5,6,7\}}.
 \end{align}
 $I'$ is then generated by all $f_{\Psi_0}'$ where $f$ is in the ideal of $\CVGMG(T_1)$. As observed before, this is the same as the ideal generated by all coefficients of monomials in the $y_I$'s. As an example let us consider $f_{\emptyset,\{1,2,\}} = x_{\emptyset}x_{\{1,2,3,4\}} - x_{\{1,2\}}x_{\{3,4\}}$. Using the relations \eqref{eq:relation1}--\eqref{eq:relation4}, we get an expression in the $z_I$'s and $y_I$'s.
 As a function on $M_{\bk,\bm} \times M_{\bm,\bl}$ it is equal to
\begin{multline}
 f_{\emptyset,\{1,2\}}(\Psi\Psi_0) = \tilde x_{\emptyset}\tilde x_{\{1,2,3,4\}} - \tilde x_{\{1,2\}} \tilde x_{\{3,4\}} \\ = (z_{\emptyset}z_{\{1, 2, 3, 5\}}-z_{\{1, 2\}}z_{\{3, 5\}})y_{\emptyset}y_{\{4, 5\}}\\
 +(z_{\emptyset}z_{\{1, 2, 3, 6\}}-z_{\{1, 2\}}z_{\{3, 6\}})y_{\emptyset}y_{\{4, 6\}}\\
 +(z_{\emptyset}z_{\{1, 2, 3, 7\}}-z_{\{1, 2\}}z_{\{3, 7\}})y_{\emptyset}y_{\{4, 7\}}\\
  +(z_{\emptyset}z_{\{1, 2, 3, 5, 6, 7\}}-z_{\{1, 2\}}z_{\{3, 5, 6, 7\}})y_{\emptyset}y_{\{4, 5, 6, 7\}}\\
 +(z_{\{5, 6\}}z_{\{1, 2, 3, 5\}}-z_{\{1, 2, 5, 6\}}z_{\{3, 5\}})y_{\{4, 5\}}y_{\{5, 6\}}\\
 +(z_{\{5, 7\}}z_{\{1, 2, 3, 5\}}-z_{\{1, 2, 5, 7\}}z_{\{3, 5\}})y_{\{4, 5\}}y_{\{5, 7\}}\\
 +(z_{\{6, 7\}}z_{\{1, 2, 3, 5\}}-z_{\{1, 2, 6, 7\}}z_{\{3, 5\}})y_{\{4, 5\}}y_{\{6, 7\}}\\
 +(z_{\{5, 6\}}z_{\{1, 2, 3, 6\}}-z_{\{1, 2, 5, 6\}}z_{\{3, 6\}})y_{\{4, 6\}}y_{\{5, 6\}}\\
 +(z_{\{5, 7\}}z_{\{1, 2, 3, 6\}}-z_{\{1, 2, 5, 7\}}z_{\{3, 6\}})y_{\{4, 6\}}y_{\{5, 7\}}\\
 +(z_{\{6, 7\}}z_{\{1, 2, 3, 6\}}-z_{\{1, 2, 6, 7\}}z_{\{3, 6\}})y_{\{4, 6\}}y_{\{6, 7\}}\\
 +(z_{\{5, 6\}}z_{\{1, 2, 3, 7\}}-z_{\{1, 2, 5, 6\}}z_{\{3, 7\}})y_{\{4, 7\}}y_{\{5, 6\}}\\
 +(z_{\{5, 7\}}z_{\{1, 2, 3, 7\}}-z_{\{1, 2, 5, 7\}}z_{\{3, 7\}})y_{\{4, 7\}}y_{\{5, 7\}}\\
 +(z_{\{6, 7\}}z_{\{1, 2, 3, 7\}}-z_{\{1, 2, 6, 7\}}z_{\{3, 7\}})y_{\{4, 7\}}y_{\{6, 7\}}\\
 +(z_{\{5, 6\}}z_{\{1, 2, 3, 5, 6, 7\}}-z_{\{1, 2, 5, 6\}}z_{\{3, 5, 6, 7\}})y_{\{5, 6\}}y_{\{4, 5, 6, 7\}}\\
 +(z_{\{5, 7\}}z_{\{1, 2, 3, 5, 6, 7\}}-z_{\{1, 2, 5, 7\}}z_{\{3, 5, 6, 7\}})y_{\{5, 7\}}y_{\{4, 5, 6, 7\}}\\
 +(z_{\{6, 7\}}z_{\{1, 2, 3, 5, 6, 7\}}-z_{\{1, 2, 6, 7\}}z_{\{3, 5, 6, 7\}})y_{\{6, 7\}}y_{\{4, 5, 6, 7\}}.
\end{multline}
Every single coefficient of a monomial in the $y_I$'s then gives a
generator for $I'$. It should be clear how to proceed in principle with
the other $f_{I,J}$'s. So Algorithm~\ref{alg:Procedure} calls itself twice,
once for $T_1$ and once for $T_2$.
\end{ex}

%\bibliographystyle{plain}
%\bibliography{diffeq}

\end{document}